\documentclass[12pt,english]{article}
\usepackage{amsfonts,amsmath,amsthm,amscd,amssymb,latexsym,amsbsy}

    \usepackage[T2A]{fontenc}
    \usepackage[cp1251]{inputenc}

\usepackage[english]{babel}
\begin{document}

\selectlanguage{english}
\fontencoding{T2A}
\selectfont

\title{Semi-scalar equivalence of polynomial matrices\thanks{ Pidstryhach Institute for Appl. 
Problems  of Mech. and Math., 
  Str. Naukova 3b, L'viv, Ukraine, 79060}}
\author{Volodymyr M. Prokip\thanks{v.prokip@gmail.com }}

\date{  }

\theoremstyle{plain}
\newtheorem{theorem}{Theorem}[section]
\newtheorem{lemma}{Lemma}[section]
\newtheorem{proposition}{Proposition}[section]
\newtheorem{corollary}{Corollary}
\newtheorem{definition}{Definition}[section]
\theoremstyle{definition}
\newtheorem{example}{Example}[section]
\newtheorem{remark}{Remark}[section]
\newcommand{\keywords}{\textbf{Keywords }\medskip}
\newcommand{\subjclass}{\textbf{MSC 2008. }\medskip}
\numberwithin{equation}{section}

\maketitle

\begin{abstract}
Polynomial $n\times n$ matrices $A(\lambda)$ and $B(\lambda)$ over a field $\mathbb F $ are 
called semi-scalar equivalent if there 
exist a nonsingular $n\times n$ matrix $P$ over the field $\mathbb F $ and an invertible $n\times n$
 matrix $Q(\lambda)$ over the ring ${\mathbb F}[\lambda]$ such that 
 $A(\lambda)=P B(\lambda)Q(\lambda).$
 The semi-scalar  equivalence of matrices over a 
field $ {\mathbb F} $ contain the problem of similarity between 
two families of matrices. 
 Therefore, these equivalences of 
matrices can be considered a difficult problem in linear algebra. 

The aim of the present paper is  to present the necessary 
and sufficient conditions of semi-scalar equivalence of nonsingular 
matrices $A(\lambda)$ and $ B(\lambda) $ over a field ${\mathbb F }$ of characteristic zero in 
terms of solutions of a homogenous system of linear equations. 
We also establish similarity of monic polynomial matrices $A(\lambda)$ and $B(\lambda)$ over a field. 
 
\end{abstract}



\section{Introduction}\label{SSEQINT}
\newcommand{\diag}{{\rm diag\,} }
 \newcommand{\rank}{{\rm rank \,}}
 Let $\mathbb F$ be a field. Denote by $M_{m,n}({\mathbb F})$ the set of 
$m\times n$ matrices over $ \mathbb F $ and by $M_{m,n}( {\mathbb F }[\lambda])$ 
the set of $m\times n$ matrices over the polynomial ring $ {\mathbb F}[\lambda]$. 
A polynomial $a(\lambda) =a_0\lambda^k+a_1\lambda^{k-1}+ \dots + a_k \in {\mathbb F }(\lambda)$ 
is said to be monic if the first non-zero term $a_0$ is equal to 1.

Let $A(\lambda) \in M_{n,n}( {\mathbb F} [\lambda ])$ be a nonzero matrix and ${\rank }A(\lambda )=r$. 
Then $A(\lambda)$ is equivalent to a diagonal matrix, i.e., there exist matrices 
$P(\lambda), Q(\lambda) \in GL(n, {\mathbb F} [\lambda])$
 such that
 $$
P(\lambda) A(\lambda) Q(\lambda)= S_A(\lambda)=
\diag \begin{pmatrix} a_1(\lambda),  a_2(\lambda),  \dots \,,  a_r(\lambda),  0, \dots \,, 0 \end{pmatrix},
$$
where $a_j(\lambda)\in {\mathbb F} [\lambda]$ are  monic polynomials for all $j=1, 2, \dots , r$ 
and $a_1(\lambda) | a_2(\lambda) | \dots | a_r(\lambda)$ (divides) are the invariant factors 
of $A(\lambda)$. 
The diagonal matrix $ S_A(\lambda)$ is called the Smith normal form of  $A(\lambda)$.

\begin{definition}\label{DefSSEqu01} (See \cite{Kaz81}, Chapter 4.) Matrices $A(\lambda), B(\lambda) \in M_{n,n}({\mathbb F}[\lambda])$ 
are said to be
semi-scalar equivalent  if there exist matrices 
$P \in GL(n, {\mathbb F})$
and $Q(\lambda) \in GL(n, {\mathbb F}[\lambda])$ such that 
$A(\lambda)=P B(\lambda)Q(\lambda).$
 \end{definition}

Let $A(\lambda) \in M_{n,n}({\mathbb F} [\lambda])$ be nonsingular matrix over 
an infinite field ${\mathbb F }$. Then 
 $A(\lambda)$ is semi-scalar  equivalent to the lower triangular matrix (see \cite{Kaz81}) 
$$S_{l}(\lambda)=\left[
\begin{array}{c c c c c} s_{11}(\lambda) & 0 & \dots & \dots  & 0\\
s_{21}(\lambda) & s_{22}(\lambda)\; & 0 & \dots & 0\\
\dots  & \dots   & \dots & \dots & \dots \\
s_{n1}(\lambda)  & s_{n2}(\lambda) & \dots & s_{n,n-1}(\lambda)  & s_{nn}(\lambda)\\
 \end{array}\right]
$$
 with the following properties: 
 \begin{enumerate}
   \item[{\rm (a)}] $s_{ii}(\lambda)=s_i(\lambda)$, $i=1, 2, \ldots, n$,
	where $s_1(\lambda) | s_2(\lambda)| \cdots
|s_n(\lambda)$ (divides) are the invariant factors of $A(\lambda)$;
   \item[{\rm (b)}] $s_{ii}(\lambda)$ divides $s_{ji}(\lambda)$ for all $i, j$ with $1 \le
i <j \le n$.
   \end{enumerate}
	
	Later, the same upper triangular form was obtained in \cite{Bar}. 
 	Let ${\mathbb F}=\mathbb{Q}$ be the field of rational numbers. Consider 
	the following examples. 
	\begin{example}\label{SSEquExam01}
	For singular matrix
 $A(\lambda)=\begin{bmatrix}
          \lambda & \lambda \\
          \lambda^2+1 & \lambda^2+1\\
          \end{bmatrix}  \in M_{2,2}(\mathbb{Q}[\lambda])$ 
					there do not exist invertible matrices 
\rule{0pt}{6mm}	$P \in M_{2,2}(\mathbb{Q})$ and 
$Q(\lambda) \in M_{2,2}(\mathbb{Q}[\lambda])$ such that 
$$
   PA(\lambda)Q(\lambda)= S_l(\lambda)=
     \begin{bmatrix}
          1         & 0   \\
          * & 0 \\
          \end{bmatrix}.
$$	
Thus, for a singular matrix $A(\lambda)$, the matrix $S_l(\lambda) $  does not always exist.		
\end{example}
	
	\begin{example}\label{SSEquExam02}
	For nonsingular matrix 
  $$A(\lambda)=\begin{bmatrix}
          1 & 0 \\
          \lambda^2-\lambda & (\lambda-1)^4 
          \end{bmatrix} \in  M_{2,2}(\mathbb{Q}[\lambda]) $$ there exist invertible matrices 
					$$P=\begin{bmatrix}
          1 & 2 \\
          -2 &  -5  \\
          \end{bmatrix} \quad \mbox{  and } \quad 
Q(\lambda)=\begin{bmatrix}
          2\lambda^2-6\lambda+5 &  2(\lambda-1)^4  \\
          -2        &  -2\lambda^2+2\lambda -1 \\
          \end{bmatrix}  $$ such that
$$
   PA(\lambda)Q(\lambda)= B(\lambda)=
     \begin{bmatrix}
          1         & 0   \\
          \lambda^2-3\lambda & (\lambda-1)^4 \\
          \end{bmatrix}.
$$

Hence, matrices $A(\lambda)$ and $B(\lambda)$ are semi-scalar equivalent. 
It is evident that   $A(\lambda)$ and $B(\lambda)$  have conditions  (a) and (b) 
for semi-scalar equivalence. Thus, the matrix $S_l(\lambda)$ is defined not 
uniquely with respect to the semi-scalar equivalence for nonsingular matrix 
 $A(\lambda)$.
\end{example}

Dias da Silva and Laffey studied polynomial matrices up to 
PS-equivalence. 
\begin{definition}\label{DefSSEqu02} (See \cite{DiasLaf}.)
 Matrices $A(\lambda), B(\lambda)\in M_{n,n}({\mathbb F}[\lambda])$ are PS-equivalent if 
$ A(\lambda )=P(\lambda )B(\lambda )Q$ 
for some $P(\lambda) \in GL(n, {\mathbb F}[\lambda])$ and $Q \in GL(n,
 {\mathbb F})$.
\end{definition}

 Let ${\mathbb F }$ be an infinite field. A matrix 
 $A(\lambda) \in M_{n,n}({\mathbb F} [\lambda])$ with $\det A(\lambda) \not= 0$
 is PS-equivalent to the upper triangular matrix (see \cite{DiasLaf}, Proposition 2) 
$$S_u(\lambda)=\left[
\begin{array}{c c c c } s_{11}(\lambda)\; & s_{12}(\lambda)\; & \dots \; & s_{1n}(\lambda)\\
0 & s_{22}(\lambda)\; & \dots \; & s_{2n}(\lambda)\\
\dots  & \dots   & \dots  & \dots \\
0 & \dots \; & 0 \; & s_{nn}(\lambda)\\
 \end{array}\right]
$$
 with the following properties: 
 \begin{enumerate}
   \item[{\rm (a)}] $s_{ii}(\lambda)=s_i(\lambda)$, $i=1, 2, \dots, n$,
	where $s_1(\lambda) | s_2(\lambda)| \cdots
|s_n(\lambda)$ (divides) are the invariant factors of $A(\lambda)$;
   \item[{\rm (b)}] $s_{ii}(\lambda)$ divides $s_{ij}(\lambda)$ for all integers $i, j$ with $1 \le
i <j \le n$;
   \item[{\rm (c)}] if $i \not= j$ and $s_{ij}(\lambda)\not=0$,
then $s_{ij}(\lambda)$ is a monic polynomial and $\deg s_{ii}(\lambda) <
\deg s_{ij}(\lambda) < \deg s_{jj}(\lambda)$.
\end{enumerate}

The matrix $S_u(\lambda)$ is called a near canonical form of the matrix $A(\lambda)$ with 
respect to PS-equivalence.  We note that 
 conditions  (a) and (b) for semi-scalar 
equivalence were proved in \cite {Kaz81}. 
It is evident that matrices 
$A(\lambda), B(\lambda) \in M_{n,n}( {\mathbb F} [\lambda])$ are PS-equivalent if and
only if the transpose matrices $A^T(\lambda)$ and $B^T(\lambda)$ are 
semi-scalar equivalent. 
It is easy to make sure that 
 the matrix $S_u(\lambda)$ is not uniquely determined for the 
nonsingular matrix $A(\lambda)$ with respect to PS-equivalence (see Example \ref{SSEquExam01}).

  It is clear that  semi-scalar  equivalence and PS-equivalence  represent an equivalence relation on 
	 $M_{n,n}( {\mathbb F} [\lambda])$.
	The semi-scalar equivalence and PS-equivalence of matrices over a  
field $ {\mathbb F} $ contain the problem of similarity between 
two families of matrices (see \cite{DiasLaf, Frid, Kaz81, KazPet, Serg2000}). 
 In most cases, these problems are involved with the classic unsolvable problem of 
a canonical form of a pair of matrices over a field with respect to simultaneous similarity. 
 At present, such problems are called wild 
(\cite {Drozd77}, \cite {Drozd80}).
Thus, these equivalences of 
matrices can be considered a difficult problem in linear algebra. 
On the basis of the semi-scalar equivalence of polynomial matrices in \cite {Kaz81} 
algebraic methods for factorization of matrix polynomials  were developed. 
We note that these equivalences were used in the study of the controllability of linear systems 
 \cite{Dodig2008}.

The problem of semi-scalar equivalence of matrices  includes  the following two problems: 
(1) the determination of a complete system of invariants and  (2) the construction of 
a canonical form for a matrix  with respect to semi-scalar equivalence. But these problems 
 have  satisfactory solutions only in isolated cases. The canonical and normal 
forms with respect to semi-scalar equivalence for a matrix 
pencil $A_0\lambda + A_1 \in M_{n,n}({\mathbb F} [\lambda])$, where $A_0$  is 
nonsingular, were investigated in 
\cite{prokip2012canonical} and \cite{prokip2013normal}. More detail about 
semi-scalar equivalence and many references to the original literature can be found in 
\cite{KazBil, Melsim93, shavar2018in}.

The paper is organized as follows. In Section \ref{SSEQPRRES} we prove
preparatory results of this article. Necessary and sufficient
conditions, under which  nonsingular 
matrices $A(\lambda)$ and $ B(\lambda) $ over a field ${\mathbb F }$ of characteristic zero
 are semi-scalar equivalence are  
proposed in Section \ref{SSEQMRES}. In  Section \ref{SSEquExam}
numerical examples are also given.

\section{Preparatory  notations and results}\label{SSEQPRRES}

To prove the main result, we need the following notations and 
propositions. Let $\mathbb F$ be a field of characteristic zero. In the  
 polynomial ring 
 ${\mathbb F}[\lambda]$ we consider the operation of differentiation ${\bf D}$.

  Let $a(\lambda)=a_0\lambda^l + a_1\lambda^{l-1}+ \ldots + a_{l-1}x +a_l \in {\mathbb F}[\lambda].$ 
  Put
$$
{\bf D}\left( a(\lambda) \right)=la_0\lambda^{l-1} + (l-1)a_1\lambda^{l-2}+ \ldots +
 a_{l-1}=a^{(1)}(\lambda) $$ and
  $$  {\bf D}^k(a(\lambda))={\bf D}(a^{(k-1)}(\lambda))=a^{(k)}(\lambda)
$$
  for every natural $ k \ge 2 $.
 The differentiation of a matrix $
  A(\lambda)=\left[\begin{array}{c} a_{ij}(\lambda) \end{array}\right] \in M_{m,n}({\mathbb F} [\lambda])
$ is understood as its elementwise differentiation, i.e., 
 $$
 A^{(1)}(\lambda)={\bf
D}(A(\lambda))= \lbrack {\bf D} (a_{ij}(\lambda)) \rbrack =\lbrack
a^{(1)}_{ij}(\lambda) \rbrack $$  and $ A^{(k)}(\lambda)={\bf
D}(A^{(k-1)}(\lambda))
$ 
is the $k$-th derivative of $A(\lambda)$ for every natural $ k\geq 2.$

   Let 
$b(\lambda)=(\lambda-\beta_1)^{k_1}(\lambda-\beta_2)^{k_2}\cdots (\lambda-\beta_r)^{k_r}\in
  \mathbb { F} [\lambda]$, ${\deg \,}b(\lambda)= k=k_1+k_2+ \dots + k_r$, and
$A(\lambda)\in M_{m,n}({\mathbb F} [\lambda])$. For the monic polynomial 
$b(\lambda)$ and the matrix $A(\lambda)$ we define the matrix 
$$
 M[A, b] = \left[ \begin{array}{c}
          N_1 \\
         \rule{0pt}{5mm} N_2 \\
          \vdots \\
          N_r \\
          \end{array} \right] \in M_{mk, n}(  {\mathbb F} ), $$
          where
  $N_j=\left[ \begin{array}{c}
          A(\beta_j) \\
          \rule{0pt}{5mm}A^{(1)}(\beta_j)\\
          \vdots \\
          A^{(k_j-1)}(\beta_j)\\
       \end{array}\right] \in M_{mk_j, n}(  {\mathbb F} ),
       $ $j=1, 2, \dots, r $.

\begin{proposition}\label{LeSSE01}
Let $b(\lambda)=(\lambda-\beta_1)^{k_1}(\lambda-\beta_2)^{k_2}\cdots
(\lambda-\beta_r)^{k_r}\in   {\mathbb F} [\lambda]$, where $\beta_i \in{\mathbb F}$ for all 
$i = 1, 2, \hdots, r$, and  $A(\lambda) \in M_{m, n}(  {\mathbb F} [\lambda])$ be a nonzero matrix.
 Then $A(\lambda)$ admits the representation
\begin{equation}
A(\lambda)=b(\lambda)C(\lambda), \label{SSEquR1}
\end{equation}
if and only if $M[A, b]=0.$
\end{proposition}

\begin{proof}  Suppose that (\ref{SSEquR1}) holds. 
 It is evident that $b(\beta_j)=b^{(1)}(\beta_j)= \ldots
=b^{(k_j-1)}(\beta_j)=0 $ for all $j=1,2, \dots , r $ and 
$A(\beta_j)=0$. Differentiating equality  (\ref{SSEquR1}) $(k_j-1)$
times and substituting each time $\lambda=\beta_j$ into both sides of 
the obtained equalities, we finally obtain 
 $$
\left[
 \begin{array}{c}
A(\beta_j)\\
\rule{0pt}{5mm} A^{(1)}(\beta_j) \\
    \rule{0pt}{5mm}      A^{(2)}(\beta_j) \\
          {\vdots} \\
          A^{(k_j-1)}(\beta_j) \\
          \end{array}\right] =
     \left[ \begin{array} {c} 0   \\
                        0         \\
          0 \\
          {\vdots} \\
          0 \\
          \end{array}\right].
$$
 Thus, $N_j=0$.
Since  $1\le j \le r$, we have $M\left[ A, \; b\right]=0$.

Conversely,  let $M\left[ A, b\right]=0$. Dividing the matrix $A(\lambda )$ by 
$I_nb(\lambda )$ with residue (see, for instance, 
Theorem 7.2.1 in the classical book by Lancaster and Tismenetski \cite{LanTis}), we have 
        $$ A(\lambda )=b(\lambda )C(\lambda )+R(\lambda ), $$
where $C(\lambda ), R(\lambda ) \in M_{m, n}( {\mathbb F} [\lambda ])$ and 
$\deg R(\lambda) < \deg b(\lambda )$. Thus, $M\left[ A, b\right]=M\left[ R, b\right]=0$. 
Since $M[R, b]=0$, then $R(\lambda)=(\lambda-\beta_i)^{k_i}R_i(\lambda)$ for all 
$i = 1, 2, \hdots, r$, i. e. $R(\lambda)=b(\lambda)R_0(\lambda)$. On the other hand, 
$\deg R(\lambda) < \deg b(\lambda)$. Thus, $R(\lambda)\equiv 0$. 
This completes the proof. \end{proof}

\begin{corollary} \label{CorSSE01} Let  $A(\lambda) \in M_{n,n}(  {\mathbb F} [\lambda])$
be a matrix of ${\rank } A(\lambda) \ge n-1$ with the Smith normal form 
     $S(\lambda)={ \diag }(s_1(\lambda),  \dots, s_{n-1}(\lambda),
     s_n(\lambda)).$ 
    If $$s_{n-1}(\lambda)=(\lambda-\alpha_1)^{k_1}(\lambda-\alpha_2)^{k_2}\cdots
    (\lambda-\alpha_r)^{k_r},$$ where $\alpha_i \in   {\mathbb F} $ for all $i=1, 2, \dots,
    r$; 
     then
     $
     M[A^*, s_{n-1}]=0 .$
\end{corollary}
\begin{proof}
 Write the matrix $A(\lambda) $ as $A(\lambda)=U(\lambda)S(\lambda)V(\lambda)$, where 
$U(\lambda), V(\lambda) \in GL(n,  {\mathbb F}[\lambda])$. Then 
$A^*(\lambda)=V^*(\lambda)S^*(\lambda)U^*(\lambda)$. Put $$d(\lambda)=s_1(\lambda)s_2(\lambda) \cdots
s_{n-1}(\lambda).$$ Since ${\rank } A(\lambda) \ge n-1 $, we have 
$A^*(\lambda) \not=0 $. It is clear that 
$$
S^*(\lambda) =
  {\diag}\left( \frac {s_n(\lambda)}{s_1(\lambda)}, \cdots,
          \frac{s_n(\lambda)}{s_{n-1}(\lambda)}, 1 \right)d(\lambda). $$

 Hence, $A^{*}(\lambda)$ admits the representation 
 $A^{*}(\lambda)=s_{n-1}(\lambda)B(\lambda)$, where $B(\lambda) \in M_{n,n}( {\mathbb F} [\lambda])$.
 By virtue of Proposition~\ref{LeSSE01}, $M[A^{*}, s_{n-1}]=0 $. This completes the proof. 
\end{proof}
 The Kronecker product of matrices 
$A=\left[a_{ij} \right]$ ($n\times m$)
and $B$ is denoted by 
$$A\otimes B= \left[ \begin{array} {c c c} a_{11}B  &
 \dots  &  a_{1m}B \\
                         \vdots &   & \vdots  \\
                        a_{n1}B & \dots  &  a_{nm}B \\
\end{array} \right].$$
Let nonsingular  matrices $A(\lambda), B(\lambda) \in
M_{n,n}( {\mathbb F}[\lambda])$  be equivalent and 
$$S(\lambda)= {\diag }(s_1(\lambda),  \dots , s_{n-1}(\lambda) , s_n(\lambda))$$ 
be their  Smith normal form. For $A(\lambda)$ and $ B(\lambda)$ we define the matrix 
$$ D(\lambda)=\Bigl(\Bigl(s_1(\lambda)s_2(\lambda)
\cdots s_{n-1}(\lambda)\Bigr)^{-1}B^*(\lambda)\Bigr)\otimes A^T(\lambda) \in
M_{n^2,n^2}(   {\mathbb F} [\lambda]).
$$
It may be noted if $S(\lambda)= {\diag }(1,  \dots , 1, s(\lambda))$ is the Smith 
normal form of the matrices $A(\lambda)$ and $ B(\lambda)$, then 
$D(\lambda)=B^*(\lambda)\otimes A^T(\lambda)$.

\section{Main results }\label{SSEQMRES}

It is clear that two semi-scalar or PS-equivalent matrices are always 
equivalent. The converse of the above statement is not always true. 
The main result of this chapter is the following theorem. 

\begin{theorem}\label{TheSSE01}
Let nonsingular  matrices $A(\lambda), B(\lambda) \in M_{n,n}( {\mathbb F}[\lambda])$
be equi\-va\-lent and $S(\lambda)= {\diag}(s_1(\lambda),  \dots \;, s_{n-1}(\lambda), 
s_n(\lambda))$ be their  Smith normal form. Further, let 
 $s_{n}(\lambda)=(\lambda-\alpha_1)^{k_1}(\lambda-\alpha_2)^{k_2}\cdots
    (\lambda-\alpha_r)^{k_r},$ where $\alpha_i \in   {\mathbb F} $ for all $i=1, 2, \dots,
    r$.
Then $A(\lambda)$ and $B(\lambda)$ 
are semi-scalar equivalent if and only if ${\rank }M[D, s_n] < n^2$  and the homogeneous 
system of equations \quad $M[D, s_n] {x}= \bar{0}$ has a solution 
${x}=[v_1, v_2, \dots, v_{n^2}]^T$ over 
$ {\mathbb F}$ such that  the matrix $$V=
\left[\begin{array}{c c c c }
v_1 & v_2 & \dots \; & \; v_n \\
v_{n+1}& v_{n+2} &\dots \; &\; v_{2n}\\
\dots & \dots & \dots \; &\; \dots \\
v_{n^2-n+1}&v_{n^2-n+2} & \dots & \;v_{n^2}\\
\end{array}\right] $$ is nonsingular. 
 If $\det V \not=0$, then $VA(\lambda)= B(\lambda)Q(\lambda)$, 
where $Q(\lambda) \in GL(n, {\mathbb F} [\lambda])$.
 \end{theorem}

\begin{proof} Let nonsingular  matrices $A(\lambda)$ and $ B(\lambda)$ in 
$M_{n,n}({\mathbb F} [\lambda])$ be  semi-sca\-lar equivalent, i.e., 
  $A(\lambda) = PB(\lambda)Q(\lambda),$  
where $P \in GL(n, {\mathbb F} )$ and $Q(\lambda) \in GL(n,  {\mathbb F} [\lambda])$. 
From the last equality we have 
\begin{equation} B^*(\lambda)P^{-1}A(\lambda) = Q(\lambda) \det B(\lambda). \label{T22} \end{equation}

 Write  $B^*(\lambda)$ 
in the form $B^* (\lambda)=d(\lambda )C(\lambda )$ (see the proof of Corollary \ref{CorSSE01})
and 
$\det B(\lambda )=b_0 d(\lambda )s_n(\lambda )$, where $d(\lambda )=s_1(\lambda )s_2(\lambda )\cdots
s_{n-1}(\lambda )$, 
 $C(\lambda) \in M_{n,n}({\mathbb F}[\lambda ])$ and $b_0$ is a nonzero element in $\mathbb F$.
Now rewrite equality (\ref{T22}) as 
$$
d(\lambda)C(\lambda)P^{-1}A(\lambda)=Q(\lambda)d(\lambda)s_n(\lambda)b_0.
$$
This implies that 
 \begin{equation}
      C(\lambda)P^{-1}A(\lambda)=Q(\lambda)s_n(\lambda)b_0. \label{T23}
 \end{equation}

 Put
 $$ P^{-1}=\left[\begin{array} {cccc}
                 v_1 & v_2 & \ldots & v_n \\
                 v_{n+1}&v_{n+2}&\dots& v_{2n}\\\dots& \dots& \dots& \dots \\
                 v_{n^2-n+1}&v_{n^2-n+2}&\dots& v_{n^2}\\
            \end{array}\right]
            $$ 
						             and
                  $$ Q(\lambda)b_0 =W(\lambda)=\left[ \begin{array}{cccc}
                  w_1(\lambda) & w_2(\lambda) & \ldots & w_n(\lambda) \\
                 w_{n+1}(\lambda)&w_{n+2}(\lambda)&\dots& w_{2n}(\lambda)\\ \dots& \dots& \dots& \dots \\
                 w_{n^2-n+1}(\lambda)&w_{n^2-n+2}(\lambda)&\dots& w_{n^2}(\lambda)\\
                             \end{array}\right],
$$
 where $v_j \in  {\mathbb F} $  and $w_j(\lambda) \in  {\mathbb F} [\lambda]$ for
all $ j=1, 2, \dots, n^2$.  Then we can write equality (\ref{T23})  in the 
form (see \cite{LanTis}, Chapter 12)  
\begin{multline}\label{EqSSE222}
\big( C(\lambda)\otimes A^T(\lambda)\big)\cdot \left[ \begin{array}{cccc}
                 v_1, & v_2, & \dots \, , & v_{n^2}
            \end{array}\right]^T = \\ s_n(\lambda) \left[\begin{array}{cccc}
         w_1(\lambda), & w_2(\lambda), & \dots \, , & w_{n^2}(\lambda)
 \end{array}\right]^T.
\end{multline}
Note that $C(\lambda)\otimes A^T(\lambda)=D(\lambda)$. 
 In view of equality (\ref{EqSSE222}) and Proposition \ref{LeSSE01}, we have 
 $
   M[D, s_n] \left[ \begin{array}{cccc}
                 v_1,\! &\! v_2,\! &\! \dots\, , &\! v_{n^2}\!
            \end{array}\right]^T =\bar{0}.$ This implies that
 $ {\rank} M[D, s_n] < n^2.$

Conversely, let ${\rank\,}M[D, s_n] < n^2$ and for matrix 
$M[D, s_n]$ there exists a vector 
$ x_0= \left[ \begin{array}{cccc}
                \! v_1,\! &\! v_2,\! &\! \dots\, , &\! v_{n^2}\!
            \end{array}\right]^T,$ 
  where  $v_j \in  {\mathbb F} $ for all $j=1, 2, \ldots, n^2 $,
such that $ M[D, s_n] x_0=\bar{0}$ and the matrix 
$$V=\left[
\begin{array}{cccc}
                 v_1 & v_2 & \cdots & v_n \\
              v_{n+1}&v_{n+2}&\dots& v_{2n}\\
              \dots& \dots& \dots& \dots \\
                 v_{n^2-n+1}&v_{n^2-n+2}&\dots& v_{n^2}\\
            \end{array} \right]
         $$ 
             is nonsingular.

Dividing the product $C(\lambda )VA(\lambda )$ by 
$I_ns_n(\lambda )$ with residue, we have 
        $$ C(\lambda )VA(\lambda )=s_n(\lambda )Q(\lambda )+R(\lambda ), $$
where $Q(\lambda ), R(\lambda )=
\left[ r_{ij}(\lambda)\right] \in M_{n,n}( {\mathbb F} [\lambda ])$ and 
$\deg R(\lambda ) < \deg s_n(\lambda )$. 
 From the last equality 
we obtain 
$$
  M[D, s_n] x_0=M[{\bf Col\,}R, s_n]=\bar{0},
  $$
	where  
	${\bf Col\,}R(\lambda)=\left[\begin{array}{c c c c c c c}
    r_{11}(\lambda ) & \dots & r_{1n}(\lambda ) & \ldots & r_{n,n-1}(\lambda)  & \dots & r_{nn}(\lambda )
            \end{array}\right] ^T.$  
		In accordance with Proposition \ref{LeSSE01}  ${\bf Col\,}R(\lambda)\equiv \bar{0}$. 
		Thus, $R(\lambda)\equiv 0$ and 
 \begin {equation}\label{EqSSE224}
C(\lambda)VA(\lambda)=s_n(\lambda)Q(\lambda). \end{equation}

Note that  $\det B(\lambda)=b_0d(\lambda)s_n(\lambda)$, where $b_0$ is a nonzero 
element in $ {\mathbb F} $. Multiplying both sides of equality 
(\ref{EqSSE224}) by $b_0d(\lambda)$, we have 
\begin{multline}\label{EqSSE225} \quad b_0d(\lambda)C(\lambda)VA(\lambda)=  B^*(\lambda)V 
A(\lambda) = \\ b_0d(\lambda)s_n(\lambda)Q(\lambda) = Q(\lambda)\det B(\lambda). \qquad
\end{multline}
 From equality (\ref{EqSSE225}) it follows  
  $
     VA(\lambda)=B(\lambda)Q(\lambda).
  $
 Passing to the determinants on both sides of this equality, we obtain 
$\det Q(\lambda)={\rm const}\, \not=0$. Since $Q(\lambda)\in GL(n,  {\mathbb F} [\lambda])$, we 
conclude that matrices $A(\lambda)$ and $ B(\lambda)$ are semi-scalar equivalent. This completes the proof. 
\end{proof}

It may be noted that nonsingular  matrices $A(\lambda), B(\lambda) \in M_{n,n}( {\mathbb F}[\lambda])$ 
are  PS-equivalent if and only if  $A(\lambda)^T$ and $ B(\lambda)^T$ are semi-scalar equivalent. 
Thus, Theorem~\ref{TheSSE01} gives the answer to the question: When are nonsingular matrices $A(\lambda)$ 
and $ B(\lambda)$  PS-equivalent?

In the future $ { {\mathbb F}}= {\mathbb C}$ is the field of complex 
numbers.

\begin{corollary}\label{CorSSEquivT2}
Let nonsingular  matrices $A(\lambda), B(\lambda) \in M_{n,n}( {\mathbb C}[\lambda])$ 
be equivalent and $S(\lambda)= {\diag}(s_1(\lambda),  \dots , s_{n-1}(\lambda),
s_n(\lambda))$ be their  Smith normal form. 
Then $A(\lambda)$ and $B(\lambda)$ 
are semi-scalar equivalent if and only if $${\rank }M[D, s_n] < n^2$$ and the homogeneous 
system of equations $M[D, s_n] x= \bar{0}$ has a solution 
$x=[v_1, v_2, \dots, v_{n^2}]^T$ over 
$ {\mathbb C}$ such that  the matrix 
$$V=
\left[\begin{array}{c c c c }
v_1 & v_2 & \dots \; & \; v_n \\
v_{n+1}& v_{n+2} &\dots \; &\; v_{2n}\\
\dots & \dots & \dots \; &\; \dots \\
v_{n^2-n+1}&v_{n^2-n+2} & \dots & \;v_{n^2}\\
\end{array}\right]$$ is nonsingular.
\end{corollary}

\begin{definition}\label{DefSteqMat} Two families of $n\times n$ matrices over the field $ {\mathbb C} $
$$
 {\bf A}=\left\{ A_1, A_2, \dots, A_r  \right\}\quad {\rm and}
   \quad {\bf B}=\left\{B_1, B_2, \ldots, B_r \right\}
   $$
are said to be similar if there exists a matrix $T\in GL(n,
 {\mathbb C} )$ such that $$A_i =TB_i T^{-1} \qquad \mbox {for all}
\qquad i=1,2, \dots, r.$$
\end{definition}
%

The families $\bf{A}$ and $\bf{B}$ we associate with monic matrix polynomials 
$$ A(\lambda)=I_n\lambda^r + A_1 \lambda^{r-1} + A_2\lambda^{r-2}+ \dots + A_r
$$
  and
  $$
   B(\lambda)=I_n\lambda^r +  B_1 \lambda^{r-1} + B_2\lambda^{r-2}+ \dots + B_r
$$
over $ \mathbb C $ of degree $r$ respectively. 
The families $ {\bf A}$ and ${\bf B}$ are similar over 
$ \mathbb C $ if and only if the matrices $ A(\lambda)$ and $B(\lambda)$ are 
semi-scalar equivalent (PS-equivalent) (see \cite{DiasLaf} and \cite{Kaz81}). From 
Theorem~\ref{TheSSE01} and Corollary \ref{CorSSEquivT2} we obtain the following corollary.

\begin{corollary}\label{CorSSEquivT3}  Let $n\times n$
monic matrix polynomials  of degree $r$
$$ A(\lambda)=I_n\lambda^r +\sum_{i=1}^r A_i \lambda^{r-i} \quad \mbox{ and} \quad
 B(\lambda)=I_n\lambda^r +  \sum_{i=1}^r B_i \lambda^{r-i}
$$
  over the field of complex 
numbers $ {\mathbb C}$ be equivalent, 
and let $$S(\lambda)= {\diag}(s_1(\lambda), \dots , s_{n-1}(\lambda), s_n(\lambda)) $$ be 
their Smith normal form.

 The families $ {\bf A}=\left\{ A_1, A_2,
\ldots, A_r
     \right\}$ and    ${\bf B}=\left\{ B_1, B_2,
       \ldots, B_r \right\} \ $
are similar over $ {\mathbb C} $ if and only if ${\rank\,}M[D, s_n] < n^2$  and the
 homogeneous 
system of equations $M[D, s_n] {x}= \bar{0}$ has a solution 
${x}=[v_1, v_2, \dots, v_{n^2}]^T$ over 
$ {\mathbb C}$ such that  the matrix $$V=
\left[\begin{array}{c c c c }
v_1 & v_2 & \dots \; & \; v_n \\
v_{n+1}& v_{n+2} &\dots \; &\; v_{2n}\\
\dots & \dots & \dots \; &\; \dots \\
v_{n^2-n+1}&v_{n^2-n+2} & \dots & \;v_{n^2}\\
\end{array}\right] $$ is nonsingular.
 If $\det V \not=0$, then $A_i =V^{-1}B_iV $ for all $
 i=1, 2, \dots, r.$
\end{corollary}

\section{Illustrative examples }\label{SSEquExam}

To illustrate Theorem \ref{TheSSE01} and Corollary \ref{CorSSEquivT3} consider the following examples.

\begin{example} \label{SSEquExam03}
Matrices
   $$A(\lambda )= \left[ \begin{array}{c r}
              1 & 0 \\
              \lambda^2 + a\lambda & \lambda^4
              \end{array}\right]\quad \mbox{ and } \quad
  B(\lambda )= \left[ \begin{array}{cc}
              1 & 0 \\
              \lambda^2 +b \lambda & \lambda^4 \\
              \end{array}\right]$$  with entries from $ {\mathbb C} [\lambda]$ are equivalent 
							 for all $a, b \in {\mathbb C}$ and 
              $S(\lambda )={\diag }(1, \lambda^4 )$ is their Smith normal form. 
							In what follows $a\not= b$.

          Construct the matrix  
\begin{multline*}
   D(\lambda )=B^*(\lambda )\otimes A^T(\lambda )= \\ \left[ \begin{array}{cccc}
              \lambda^4  & \lambda^6+a\lambda^5  & 0 & 0  \\
           \rule{0pt}{5mm}   0  & \lambda^8 & 0 & 0  \\
           \rule{0pt}{5mm} -(\lambda^2 + b\lambda)   & -(\lambda^4+(a+b)\lambda^3 +ab\lambda^2)
						               & 1  & \lambda^2+ a\lambda  \\
           \rule{0pt}{5mm}   0    & -(\lambda^6 + b\lambda^5)   & 0  & \lambda^4 \\
              \end{array}\right]
\end{multline*}
 and solve the system of equations $M[D, s_2]{x}=\bar{0}$. 
 From this  it follows 
  $$ \left[ \begin{array}{rrcc}
              0   & 0   & 1 & 0  \\
              -b   & 0   & 0 & a  \\
              -2  & -2ab   & 0 & 2  \\
							0   & -6(a+b)   & 0 & 0  \\
                      \end{array}\right]
              \left[ \begin{array}{cccc}
                  v_1 \\ v_2 \\ v_3 \\ v_4
               \end{array}\right] =
                 \left[  \begin{array}{cccc}
                       0 \\ 0 \\ 0 \\ 0
                \end{array}\right]  .
$$
 From this we have, if $a+b\not=0$, then  $ A(\lambda )$ and $B( \lambda )$   
are not semi-scalar equivalent. If $a+b=0$, then $b=-a$ and system of   
equations $M[D, s_2]{x}=\bar{0}$ is solvable. 
 The vector \rule{0pt}{6mm} $\left[  \begin{array}{cccc}
                       1, &   \frac{2 }{a^2}, & 0, &-1
                \end{array}\right]^T$ is a solution of $M[D, s_2]{x}=\bar{0}$ for
								arbitrary  $a\not = 0$. \rule{0pt}{6mm}
 Thus, the matrix  $V=
                \left[ \begin{array}{cc}
                  1 &  \frac{2 }{a^2} \\
                \rule{0pt}{5mm}   0 & -1 \\
                \end{array}\right]
                $ \rule{0pt}{8mm}
			is nonsingular.

So, \rule{0pt}{5mm} if $a\not=0$ and $b = -a$,  then $A(\lambda )=\left[ \begin{array}{c r}
              1 & 0 \\
              \lambda^2 + a\lambda & \lambda^4
              \end{array}\right]$ 
and $B(\lambda )= \left[ \begin{array}{cc}
              1 & 0 \\
              \lambda^2 -a \lambda & \lambda^4 \\
              \end{array}\right]$ are semi-scalar equi\-va\-lent, i.e.,  
$A(\lambda ) = PB(\lambda)Q(\lambda),$ 
where  \rule{0pt}{6mm} 
$$P=V^{-1}=
                \left[ \begin{array}{cc}
                  1 &  \frac{2 }{a^2}\\
                 \rule{0pt}{5mm} 0 & -1 \\
                \end{array}\right]$$ and
$$Q(\lambda)=
\left[ \begin{array}{cc}
              \frac{2 \lambda^2}{a^2}+ \frac{2 \lambda}{a} +1 & \frac{2\lambda^4}{a^2} \\
             \rule{0pt}{5mm}  -\frac{2 }{a^2} &-\frac{2 \lambda^2}{a^2}+\frac{2 \lambda}{a}-1
              \end{array}\right] \in GL(2, {\mathbb C} [\lambda ]).$$
 \end{example}

\begin{example} \label{SSEquExam04} Let 
 $$
 {\bf A}=\left\{ A_1=\left[ \begin{array}{c c}
              -3 & 0 \\
              -4 & 1
              \end{array}\right], \, A_2=\left[ \begin{array}{c r}
              1 & 1 \\
              1 & 1
              \end{array}\right]  \right\}$$  and 
   $${\bf B}=\left\{B_1=\left[ \begin{array}{c c}
              1 & 0 \\
              -4 & -3
              \end{array}\right], \, B_2 =\left[ \begin{array}{c r}
              0 & 0 \\
              1 & 2
              \end{array}\right] \right\}
   $$
be two families of $2\times 2$ matrices over the field $ {\mathbb C}.$ 
Monic matrix polynomials 
   $$A(\lambda )=I_2\lambda^2+A_1\lambda+A_2=\left[ \begin{array}{c c}
              \lambda^2-3\lambda +1 & 1\\
              -4\lambda +1& \lambda^2 +\lambda +1
              \end{array}\right] $$  and  
  $$B(\lambda )= I_2\lambda^2+B_1\lambda+B_2=\left[ \begin{array}{cc}
              \lambda^2 +\lambda  & 0 \\
              -4\lambda +1& \lambda^2 -3\lambda +2 \\
              \end{array}\right]$$  with entries from $ {\mathbb C} [\lambda]$ are equivalent 
							 and 
              $S(\lambda )={\diag }(1, (\lambda^2 -1)(\lambda^2 -2\lambda) )$ is their Smith normal form. 
							It may be noted that $s_1(\lambda)=1$ and $s_2(\lambda )= (\lambda^2 -1)(\lambda^2 -2\lambda))$.
							
 Construct the matrix 
\begin{multline*} 
   D(\lambda )=B^*(\lambda )\otimes A^T(\lambda )= \\ 
	  \rule{0pt}{10mm} \left[ \begin{array}{cc}
              \lambda^2 -3\lambda+2  & 0 \\
              4\lambda -1& \lambda^2 +\lambda \\
              \end{array}\right] \otimes   
									\left[ \begin{array}{c c}
              \lambda^2-3\lambda +1 & -4\lambda +1\\
              1 & \lambda^2 +\lambda +1
              \end{array}\right]= \\ 
												\\	
										\rule{0pt}{15mm}							\left[ \begin{array}{rc}\!
							\! \lambda^2 -3\lambda+2\left[ \begin{array}{c c}
              \lambda^2-3\lambda +1\! & \!-4\lambda +1\\
              1 & \! \lambda^2 +\lambda +1
              \end{array}\right] \! &\! \left[ \begin{array}{c c}
              0 & 0\\
              0 & 0
              \end{array}\right] \\
							 & \\ 
							\! 4\lambda -1\left[ \begin{array}{c c}
              \! \lambda^2-3\lambda +1 & -4\lambda +1\\
              1 & \!\lambda^2 +\lambda +1
              \end{array}\right]& \lambda^2 +\lambda\left[ \begin{array}{c c}
              \!\lambda^2-3\lambda +1 & -4\lambda +1\\
              1 &\! \lambda^2 +\lambda +1
              \end{array}\right] \! \\
							\!\end{array}\! \right] 
							\end{multline*} 
					and solve the system of equations $M[D, s_2]{x}=\bar{0}$. 
					Crossing out zero rows in the matrix $ M [D, s_2] $ 
					and after elementary transformations over the rows of this matrix 
					we get the following system of linear equations
  $$\left[ \begin{array}{c c c c}
              1 & 1  & 0 & 0 \\
							3 & 9 & 2 & 6 \\
              7 & 49 & 6 & 42 \\
              \end{array}\right] 
							\left[ \begin{array}{c}
							x_1 \\ x_2 \\ x_3\\ x_4 \\
							\end{array}\right] =\left[ \begin{array}{c}
							0 \\0 \\0 \\0 \\
							\end{array}\right].$$
							From this system of equations  we obtain  
							$x_1=-x_2=t$, $x_3=0$ і $x_4=t$.
		The matrix   $V=
                \left[ \begin{array}{cc}
                  t &  -t \\
                   0 & t \\
                \end{array}\right]
                $ 
			is nonsingular for nonzero $t\in {\mathbb C}$. 
			Thus, the monic matrix  polynomials $A(\lambda)$ and $B(\lambda)$ are 
			are semi-scalar equivalent. Hence, families of matrices $ {\bf A} $ and $ {\bf B} $
are similar, i.e.,  
									$A_i=V^{-1}B_iV$, $i=1, 2.$ 

\end{example}


\end{document}